\documentclass[11pt]{amsart}

\usepackage[pdftex]{graphicx} 
\usepackage[usenames,dvipsnames]{xcolor}

\usepackage[english]{babel} 
\usepackage[utf8]{inputenc}

\usepackage[centering,margin=3cm]{geometry}

\usepackage{amsfonts}
\usepackage{amsmath}
\usepackage{amssymb}
\usepackage{amsthm}

\usepackage{paralist}
\usepackage[colorlinks,citecolor=magenta,urlcolor=magenta,pagebackref=true,pdftex]{hyperref}
\usepackage{soul}
\usepackage{bbm}

\usepackage[T1]{fontenc}

\usepackage[percent]{overpic} 

\theoremstyle{plain}

\newtheorem*{thm*}{Theorem}
\newtheorem{thm}{Theorem}[section]
\newtheorem{cor}[thm]{Corollary}
\newtheorem{lem}[thm]{Lemma}
\newtheorem*{lem*}{Lemma}
\newtheorem{prop}[thm]{Proposition}

\theoremstyle{definition}
\newtheorem{dfn}[thm]{Definition}

\newcommand\Defn[1]{\textbf{\color{black}#1}}

\newcommand{\R}{\mathbb{R}}
\newcommand{\Z}{\mathbb{Z}}

\newcommand{\Zn}{\mathbb{Z}^{n}}

\newcommand{\Pord}{\mathcal{O}}

\newcommand{\Pdir}{\mathcal{D}}
\newcommand{\Pbid}{\mathcal{B}}

\newcommand\ds{\mathbf{d}}
\newcommand\ext[1]{\widehat{#1}}

\DeclareMathOperator{\conv}{conv}

\DeclareMathOperator{\relin}{relint}
\newcommand{\relint}[1]{\relin \left( #1 \right) }

\title{On degree sequences of undirected, directed, and bidirected graphs}

\author{Laura Gellert}
\address{Institut f\"ur Optimierung und Operations Research, Universit\"at Ulm, Ulm, Germany}
\email{laura.gellert@uni-ulm.de}

\author{Raman Sanyal}
\address{Institut f\"ur Mathematik, Goethe-Universit\"at Frankfurt, Germany}
\email{sanyal@math.uni-frankfurt.de}


\date{\today}

\begin{document}

\begin{abstract}
    Bidirected graphs generalize directed and undirected graphs 
    in that edges are oriented locally at every node. 
    The natural notion of the degree of a
    node that takes into account (local) orientations is that of 
    net-degree. In this paper, we extend the following four topics from
    (un)directed graphs to bidirected graphs:
    \begin{compactitem}[\hspace{2em}--\ ]
        \item Erd\H{o}s--Gallai-type results: characterization of net-degree
            sequences,
        \item Havel--Hakimi-type results: complete sets of degree-preserving
            operations,
        \item Extremal degree sequences: characterization of uniquely
            realizable sequences, and 
        \item Enumerative aspects: counting formulas for net-degree sequences.
    \end{compactitem}
    To underline the similarities and differences to their (un)directed
    counterparts, we briefly survey the undirected setting and we give a
    thorough account for digraphs with an emphasis on the discrete geometry of
    degree sequences. In particular, we determine the tight and uniquely
    realizable degree sequences for directed graphs.
\end{abstract}

\maketitle

\section{Introduction}\label{sec:intro}

Let $G = (V,E)$ be an undirected simple graph. The degree $\ds(G)_v$ of a node
$v \in V$ is the number of edges $e \in E$ incident to $v$. The \Defn{degree
sequence} of $G$ is the vector $\ds(G) \in \Z^V$. Degree sequences of
undirected graphs are basic combinatorial statistics that have received
considerable attention (starting as early as 1874 with Cayley~\cite{Cay1874})
that find applications in many areas such as communication networks,
structural reliability, and stereochemistry; c.f.~\cite{application}.  A
groundbreaking result, due to Erd\H{o}s and Gallai~\cite{EG}, is a
characterization of degree sequences among all sequences of nonnegative
integers. In a complementary direction, Havel~\cite{Havel} and
Hakimi~\cite{Hakimi} described a complete set of  degree preserving operations
on graphs.  These two fundamental results have led to a thorough understanding
of degree sequences.  A geometric perspective on degree sequences, pioneered
by Peled and Srinivasan~\cite{PS}, turns out to be particularly fruitful.
Here, degree sequences of graphs on $n$ labelled nodes are identified with
points in $\Z^n$ and the collection of all such degree sequences can be
studied by way of \emph{polytopes of degree sequences}. Among the many results
obtained from this approach, we mention the formulas for counting degree
sequences due to Stanley~\cite{ZGDS} and the classification of uniquely
realizable degree sequences by Koren~\cite{Kor}; see
Section~\ref{sec:undirected}.

In this paper, we consider notions of degree sequences for other classes of
graphs. A \Defn{directed graph} $D = (V,A)$ is a simple undirected graph
together with an orientation of each edge, that is, every edge (or arc) is an
ordered tuple $(u,v)$ and is hence is oriented from $u$ to $v$.  The natural
notion of degree that takes into account orientations is the \Defn{net-degree}
of a node $v$ defined by
\begin{equation}\label{eqn:netdeg}
    \ds(D)_v \ := \ |\{ u \in V : (u,v) \in A \}| - |\{ u \in V : (v,u) \in A
    \}|.
\end{equation}
The net-degree is sometimes called the \emph{imbalance};
see~\cite{Mubayi,Pirzada}.

Of particular interest in this paper will be the class of bidirected
graphs. Bidirected graphs were introduced by Edmonds
and Johnson~\cite{edmonds} in 1970 and studied under the name of oriented
signed graphs by Zaslavsky~\cite{ZSGaG}. 
\begin{dfn} \label{dfn:bigraph}
    A \Defn{bidirected graph} or \Defn{bigraph}, for short, is a pair $B =
    (G,\tau)$ where $G = (V,E)$ is a simple undirected graph and $\tau$
    assigns to every pair $(v,e) \in V \times E$ 
    a local orientation $\tau(v,e) \in \{-1,1\}$ if  $v$ is incident to $e$
    and $\tau(v,e)=0$ otherwise.
\end{dfn}    
We call $|B| := G$ the \Defn{underlying graph} of $B$ and $\tau$ its
\Defn{bidirection function}. The local orientation of an edge $e = uv$ can be
interpreted such that $e$ points locally towards $v$ if $\tau(v,e) = 1$ and
away from $v$ otherwise. Graphically, we draw the edges with two arrows, one
at each endpoint; see Figure~\ref{fig:bigraph}. The notion of net-degree has a
natural extension to bigraphs where the net-degree of a node $v$ is given by
$\ds(B)_v = \sum_{e\in E} \tau(v,e)$.

\begin{figure}[ht]
\centering
\includegraphics[width=0.2\textwidth]{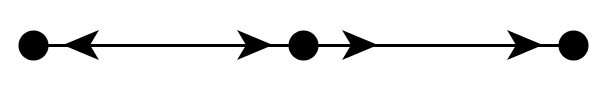} 

\caption{A bigraph $B$ on three nodes.}
\label{fig:bigraph}
\end{figure}

The objective of this paper is to extend the following four topics 
from (un)directed graphs to bigraphs:
\begin{compactitem}[\hspace{.5cm}$\circ$]
\item Erd\H{o}s--Gallai-type results: Characterization of net-degree
    sequences.
\item Havel--Hakimi-type results: Complete sets of degree-preserving
    operations.
\item extremal degree sequences: Characterization of uniquely realizable
    sequences.
\item enumerative aspects: Counting formulas for net-degree sequences.
\end{compactitem}

Bigraphs simultaneously generalize undirected and directed graphs which makes
the study of these four topics particularly interesting.  As we will show in
Section~\ref{sec:bidirected}, the characterization of net-degree sequences is
simpler for bigraphs as compared to graphs and digraphs but the
degree-preserving operations are more involved and combine those for the
undirected and directed case.

The paper is organized as follows. In \emph{Section~\ref{sec:undirected}} we
survey the results for undirected graphs. Our point of view will be more
geometric by highlighting the polytope of degree sequences introduced by
Stanley~\cite{ZGDS}.  This point of view allows for a simple treatment of
extremal degree sequences such as uniquely realizable and tight
sequences, i.e., degree sequences that attain equality in the
Erd\H{o}s--Gallai conditions.  In \emph{Section~\ref{sec:directed}} we give a
coherent treatment of net-degree sequences for directed graphs. Again, our
approach draws from the geometry of net-degree sequences. We introduce the
classes of weakly-split digraphs and width-$2$ posets and show that they
correspond to boundary and uniquely realizable digraphs, respectively. In
particular, we give explicit formulas for the number of width-$2$ posets.
\emph{Section~\ref{sec:bidirected}} is dedicated to bigraphs and we completely
address the four topics above.

For sake of simplicity, we restrict to graphs without loops and parallel edges
but all results have straightforward generalizations to $q$-multigraphs.  Our
aim was to give a self-contained treatment of results. For further background
on graphs and polytopes we refer to the books of Diestel~\cite{Die} and
Ziegler~\cite{Zie}, respectively.

\textbf{Acknowledgements.} We thank the referees for carefully reading our
paper and for making valuable suggestions.
The first author was supported by an \emph{Eva
Wolzendorf Scholarship} at Freie Universit\"{a}t Berlin.  The second author
was supported by the European Research Council under the European Union's
Seventh Framework Programme (FP7/2007-2013) / ERC grant agreement
n$^\mathrm{o}$~247029.

\section{Undirected graphs}\label{sec:undirected}

Erd\H{o}s and Gallai~\cite{EG} gave a complete characterization of degree
sequences of simple graphs: A vector $\ds=(d_1,...,d_n) \in \Z^n_{\ge 0}$ is
the degree sequence of an undirected graph $G$ on vertex set $[n] := \{
1,\dots,n \}$ if and only if 
\begin{align}
    d_1 + d_2 + \cdots + d_n & \text{ is even, and } \label{eqn:d_even}\\
	\sum_{ i \in S} d_i  - \sum_{i \in T} d_i 
	\ \leq \ & |S|(n - |T|-1) \text{ for all disjoint sets } S,T \subseteq [n].
    \label{eqn:EGcondition}
\end{align}

Geometrically, this result states that degree sequences of undirected graphs
on node set $[n]$ are constrained to an intersection of finitely many affine
halfspaces corresponding to~\eqref{eqn:EGcondition}, that is, a convex
polyhedron.  Peled and Srinivasan~\cite{PS} pursued this geometric perspective
and investigated the relation to the polytope obtained by taking the convex
hull of all the finitely many degree sequences in $\R^n$. The drawback of
their construction was that not all integer points in this polytope
corresponded to degree sequences. This was rectified by Stanley~\cite{ZGDS}
with the following modification. For a vector $\ds \in \Z^n$, let us write
$\ext{\ds} = (\ds, \frac{1}{2}(d_1+\cdots+d_n))$. The (extended) polytope of
degree sequences is
\[
    \Pord_n \ := \ \conv \Bigl\{ \ext{\ds}(G) : G = ([n],E) \text{
    simple graph}  \Bigr\} \subset \R^{n+1}.
\]

Write $[p,q]$ for the line segment connecting the points $p,q \in \R^n$. A
polytope $P \subset \R^n$ is a \Defn{zonotope} if it is the pointwise vector
sum of
finitely many line segments.  Stanley's main insight was that $\Pord_n$
is in fact a zonotope. Let $e_1,\dots,e_n$ be the standard basis for $\R^n$.

\begin{lem}[{\cite{ZGDS}}]\label{lem:ord_zono}
    For $n \ge 0$,
    \[
        \Pord_n \ = \ \sum_{1 \le i < j \le n} [0,\ext{e}_i + \ext{e}_j].
    \]
\end{lem}

These particular zonotopes (and their associated hyperplane arrangements)
occur in the study of root systems and signed graphs (see~\cite{Zas81}) which
yields that $\Pord_n$ is given by all points $p \in \R^n$ satisfying the
Erd\"os-Gallai conditions. This line of thought can be completed to a proof of
the Erd\"os--Gallai theorem by showing that for every point $\ext{\ds} \in
\Pord_n \cap \Z^{n+1}$ there is a set $E \subset \binom{[n]}{2}$ such that
\[
\ext{\ds} \ = \ \sum_{ij \in E} \ext{e}_i + \ext{e}_j,
\]
which then shows that $\ds$ is the degree sequence of $G = ([n],E)$.

This makes the treatment of degree sequences amendable to the geometric
combinatorics of zonotopes. In particular, the number of degree sequences of
simple graphs on $n$ vertices is exactly $|\Pord_n \cap \Z^{n+1}|$.  This also
furnishes a generalization of the Erd\H{o}s--Gallai conditions to multigraphs
as obtained by Chungphaisan~\cite{Chung} by purely combinatorial means. From
the geometric perspective, the generalization states that the lattice points
in $q\Pord_n$ bijectively correspond to degree sequences of loopless
$q$-multigraphs. For this, the key observation is that lattice zonotopes are
\emph{normal}, that is, every $\ds \in q\Pord_n \cap \Z^{n+1}$ is the sum of $q$
lattice points in $\Pord_n$.

Using the enumerative theory of lattice points in lattice polytopes,
i.e.~\emph{Ehrhart theory} (see \cite{beck}), it is possible to get
exact counting formulas. A graph $H = ([n],E)$ is a \Defn{quasi-forest} if
every component of $H$ has at most one cycle and the cycle is odd. Define
\[
    h(n,i) \ := \ \sum_{H} \max(1,2^{c(H)-1}),
\]
where the sum ranges over all quasi-forests $H$ on $[n]$ with $i$ edges and $c(H)$ denotes the number of cycles in $H$.  For
the special zonotope $\Pord_n$, Stanley used Ehrhart theory to show the
following.

\begin{cor} [{\cite{ZGDS}}]\label{cor:deg_count} 
    The number of degree sequences of simple graphs with node set $[n]$
    is  $\sum_i h(n,i)$.
\end{cor}

A degree sequence $\ds$ is called \Defn{tight} if $\ds$ attains equality in
one of the Erd\H{o}s-Gallai conditions, i.e., if $\ext{\ds}$ is in the boundary
of $\Pord_n$. A graph $G = ([n],E)$ is \Defn{split} if there is a partition $V
= V_c \uplus V_i$ such that $V_c$ is a clique and $V_i$ is independent.  We
call $G$ \Defn{weakly split} if there is a partition $V = V_c \uplus V_i \uplus
V_o$ such that the restriction to $V_c \cup V_i$ is a non-empty split graph,
every node of $V_o$ is adjacent to every node of $V_c$ but to no node of
$V_i$.
\begin{thm}\label{thm:undir_split}
    A degree sequence $\ds$ is tight if and only if it is the degree sequence
    of a weakly-split graph.  The number of tight degree sequences of simple graphs on $n$ vertices is 
    \[
        2 \sum_{n-i \text{ odd}}  h(n,i).
    \]
\end{thm}
\begin{proof}
    Let $\ds \in \Pord_n$ be a degree sequence and let $G=([n],E)$ be one of
    its realizations. It follows from the Erd\H{o}s--Gallai
    condition~\eqref{eqn:EGcondition} that $\ds$ is tight if and only if there
    is a partition $V_c \uplus V_i \uplus V_o$ of the node set such that 
    \begin{equation}\label{eqn:ordinary_weak-split}
    	|V_c|(|V_c|-1) + |V_c||V_o| 
    	\ = \ \sum_{ i \in V_c} d_i  - \sum_{i \in V_i} d_i. 
    \end{equation}

    Let us define $N_{ab}$ for the number of edges between $V_a$ and $V_b$ for
    $a,b \in \{i,c,o\}$. Then the right hand side can be written as
    \[
        2N_{cc} + N_{ci} +N_{co} - (2N_{ii}+N_{ic}+N_{io}) \ =  \ 2 N_{cc} +
        N_{co}- 2N_{ii} - N_{io}.
	\]
    Hence, $G$ satisfies~\eqref{eqn:ordinary_weak-split} if only if
    $N_{cc}$ and $N_{co}$ are maximal and $N_{ii}$ as well as $N_{io}$ are
    minimal. That is, $2 N_{cc} = |V_c|(|V_c|-1)$, $N_{co} = |V_c||V_o|$, and
    $N_{ii}=N_{io}=0$. This, however, is the case if and only if $G$ is
    weakly split.

    Geometrically, a sequence $\ds$ is weakly split if and only if $\ds \in
    \partial \Pord_n \cap \Z^n$. It follows from Ehrhart theory
    (see~\cite{beck}) that 
    \[
        |\relint{\Pord_n} \cap \Z^n| \ = \ h(n,n) - h(n,n-1) + \cdots + (-1)^n
        h(n,0),
    \]
    where $\relint{\Pord_n}$ denotes the relative interior, i.e., the points
    in the interior of $\Pord_n$ relative to the full-dimensional embedding in
    its affine hull.  Hence, together with Corollary~\ref{cor:deg_count}, we
    get
    \[
        |\partial\Pord_n \cap \Z^n | \ = \
        |\Pord_n \cap \Z^n | -
        |\relint{\Pord_n} \cap \Z^n | \ = \
        2 \sum_{n-i \text{ odd}}  h(n,i).
        \qedhere
    \]
\end{proof}

Havel~\cite{Havel} and Hakimi~\cite{Hakimi} showed that any two simple graphs
on the same node set with the same degree sequence can be transformed into
each other via a finite sequence of $2$-switches: In a graph $G=(V,E)$
containing four vertices $u,v,w,x\in V$ with $uv, wx \in E$ and $uw, vx \notin
E$, the operation replacing  $uv$ and $wx$ by $uw$ and $vx$ is called a
\Defn{$\boldsymbol 2$-switch} or \Defn{$\boldsymbol \Sigma$-operation} (see Figure~\ref{Figure2Switch}).

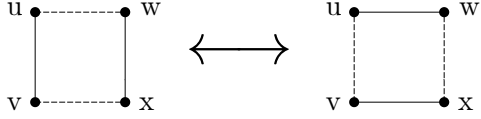
\begin{figure}[ht] 
\centering
\begin{overpic}[width=.1\textwidth]{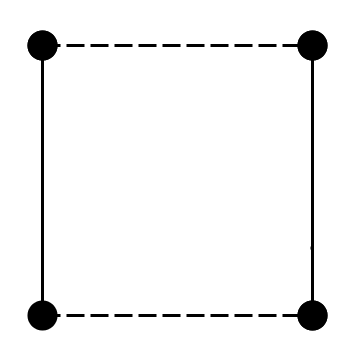}
\put(-12,85){u}
\put(-12,5){v}
\put(100,85){w}
\put(99,5){x}
\end{overpic}
\quad
\raisebox{.7cm}
{
	\scalebox{2.0}{$\longleftrightarrow$}
}
\quad
\begin{overpic}[width=.1\textwidth]{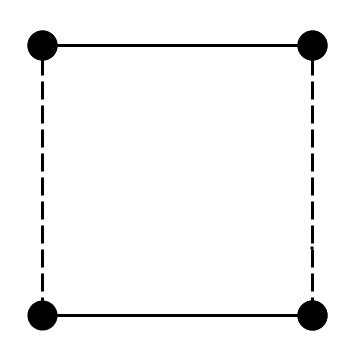}
\put(-12,85){u}
\put(-12,5){v}
\put(100,85){w}
\put(99,5){x}
\end{overpic}
\caption{Example of a $2$-switch or $\Sigma$-operation.}
\label{Figure2Switch}
\end{figure}

A degree sequence $\ds$ is called \Defn{uniquely realizable} if there is a
unique graph $G = ([n],E)$ with $\ds = \ds(G)$, that is, if no $2$-switch can
be applied to $G$. From the results of Havel--Hakimi, it follows that the
uniquely realizable degree sequences are exactly those coming from threshold
graphs. A graph $G=([n],E)$ is a \Defn{threshold graph} if $G$ has a single
node or there is a node $v \in [n]$ such that either $v$ is isolated or
connected to all remaining vertices and $G - v$ is a threshold graph.

The vertices of the $\binom{n}{2}$-dimensional cube $[0,1]^{\binom{n}{2}}$ are
the characteristic vectors of subsets of $\binom{[n]}{2}$ and hence are in
bijection with simple graphs on the node set $V = [n]$. If $\chi_E \in
\{0,1\}^{\binom{[n]}{2}}$ represents $G = ([n],E)$, then its degree sequence
is the image under the linear projection $\pi$ that takes $e_{ij}$ to
$\ext{e}_i + \ext{e}_j$ and $\Pord_n = \pi([0,1]^{\binom{n}{2}})$.  Thus, $\ds$
is uniquely realizable if $\pi^{-1}(\ext{\ds}) \cap [0,1]^{\binom{n}{2}}$
contains a unique lattice point. It is easy to see that this happens when
$\ext{\ds}$ is a vertex of $\Pord_n$. Vertices of $\Pord_n$ are easy to
describe.  The sequence $\ext{\ds}(G)$ is a vertex if and only if there exists a
vector $c=(c_1, c_2, \ldots,c_n)\in \R^n$ such that $ij$ is an edge of $G$ if
and only if $c_i + c_j > 0$ for every $1 \le i < j \le n$ (see e.g.
\cite[Section 2]{PS}). For undirected graphs, it turns out that $\ds$ is
uniquely realizable if and only if $\ext{\ds}$ is a vertex of $\Pord_n$, as was
shown by Koren~\cite{Kor}. As we will see in the next section, this is special
to the undirected situation.  The number of vertices (and hence the number of
threshold graphs) was determined by Beissinger and Peled~\cite{BeissingerPeled},
the number of all faces is given in~\cite{ZGDS}.

\section{Directed graphs} \label{sec:directed}

Let $D,D'$ be two digraphs on the node set $[n]$ that differ only by a single
directed edge connecting nodes $i$ and $j$. Then the difference of their
degree sequences will be $\pm(e_i - e_j)$. Following the geometric approach
laid out in the previous section, we are led to consider the zonotope
\[
    \Pdir_n \ := \ \sum_{1 \le i < j \le n } [e_i - e_j,e_j-e_i].
\]
This is a polytope that is full-dimensional in the $n-1$ dimensional linear
subspace given by all $\ds \in \R^n$ with $d_1 + \cdots + d_n = 0$.  It is clear
from the definition that it contains all degree sequences of directed graphs
$D = ([n],A)$. The polytope is invariant under permuting coordinates and, in
fact, it is the convex hull of all permutations of the point
\begin{equation} \label{Eq_q-vector}
    (-n+1, -n+3, -n+5, \ldots, n-5, n-3, n-1).
\end{equation}
This shows that $\Pdir_n = 2 \Pi_{n-1} - (n+1) \mathbf{1}$ where $\Pi_{n-1}$
is the well-known permutahedron; see~\cite{Zie}. In particular, it is
straightforward to show that for every $\ds \in \Pdir_n \cap \Z^n$, there is
an $A \subset [n] \times [n]$ such that 
\[
    \ds \ = \ \sum_{(i,j) \in A} e_i - e_j
\]
and $A$ gives rise to a simple digraph $D$ with $\ds = \ds(D)$. The vertices
of $\Pdir_n$ are in bijection with the net-degree sequences of tournaments on
$n$ vertices.

An inequality description of $\Pi_{n-1}$ is well-known
(cf.~\cite[Example~9.8]{Zie} or part (iii) of the theorem below) and we
conclude an Erd\H{o}s--Gallai-type result for directed graphs.
\begin{thm}\label{thm:dErdGal}
    For $\ds \in \Z^n$ the following conditions are equivalent:
    \begin{enumerate}[\rm (i)]
        \item There is a (simple) digraph $D = ([n],A)$ with net-degree
            sequence $\ds = \ds(D)$;
        \item $\ds \in \Pdir_n$;
        \item $\ds$ satisfies $d_1 + \cdots + d_n = 0$ and 
            \[
                \sum_{i \in I} d_i \ \leq \  |I|(n-|I|) 
            \]
            for all $\emptyset \neq I \subsetneq [n]$.
    \end{enumerate}
\end{thm}

For $\ds$ in non-increasing order, Theorem~\ref{thm:dErdGal}(iii) yields the
conditions found by Pirzada et al.~\cite{Pirzada}.

A simple digraph $D=(V,A)$ is \Defn{weakly split} if there is a partition $V =
V_s \uplus V_t$  such that $V_s \times V_t \subseteq A$. This class of
digraphs gives us a characterization of the tight net-degree sequences.

\begin{prop} \label{prop:Pdir_tight_weak-split}
    A vector $\ds \in \Zn$ is contained in the boundary of $\Pdir_n$ if and
    only if it is the degree sequence of a weakly split digraph.
\end{prop}

\begin{proof}
    By Theorem~\ref{thm:dErdGal}(iii), $\ds \in \partial\Pdir_n$ 
    if and only if 
    \[
        \sum_{i \in V_t}d_i \ = \  |V_t|(n-|V_t|) 
    \]
    for some nonempty subset $V_t \subset [n]$. Equality holds for a
    realization $D = ([n],A)$ if and only if $V_s \times V_t \subseteq A$ for
    $V_s = [n] \setminus V_t$.
\end{proof}

Using Exercise~4.64 in~\cite{EC1} together with Theorem~\ref{thm:dErdGal}, we
obtain the following enumerative results regarding net-degree sequences. We
write $f(n,i)$ for the number of forests on $n$ labelled vertices with $i$
edges.


\begin{cor}
    The number of net-degree sequences of digraphs on $n$ labelled nodes is
    \[
        \sum_{i=0}^{n-1} 2^i f(n,i).
    \]
    The number of degree sequences of weakly-split digraphs on $n$ vertices
    equals
    \[
    \sum_{i=0}^{ \frac{n-2}{2} } 2^{2i+1} f(n,2i) \quad \text{if
    $n$ is even} \qquad \text{ and } \qquad
    \sum_{i=1}^{\frac{n-1}{2}} 2^{2i} f(n,2i-1) \quad \text{if
    $n$ is odd}.
    \]
\end{cor}
\begin{proof}
    By Exercise~4.64 of~\cite{EC1}, the number of lattice points in $m \cdot
    \Pi_{n-1}$ for $m \ge 0$ is given by 
    \[
        E_{n-1}(m) \ = \  f(n,n-1)m^{n-1} + f(n,n-2) m^{n-2}\cdots + f(n,0)\,.
    \]
    Since $\Pdir_n$ and $2\Pi_{n-1}$ differ only by a lattice translation, the
    number of net-degree sequences is $E_{n-1}(2)$.
    By Ehrhart--Macdonald reciprocity (see~\cite[Ch.~4]{beck}), the number of
    lattice points in the relative interior of $m \Pi_{n-1}$ for $m>0$ is
    $(-1)^{n-1}E_{n-1}(-m)$. Hence, the number of degree sequences of
    weakly-split digraphs is $E_{n-1}(2) - (-1)^{n-1}E_{n-1}(-2)$ and the
    result follows.
%
\end{proof}
Two obvious operations that leave the net-degree sequence of a digraph $D =
(V,A)$ unchanged are the following: A \Defn{$\boldsymbol \Delta$-operation}
either adds a directed triangle to three pairwise non-adjacent nodes or
removes a directed triangle.  Likewise, we may replace an oriented edge
$(u,w)$ with two edges $(u,v),(v,w)$ where $v \in V$ is a node not adjacent to either
$u$ or $w$. This or its inverse is called a \Defn{$\boldsymbol
\Lambda$-operation}. See Figure~\ref{fig:dir-ops} for an illustration of both
operations. The following result is taken from~\cite{Pirzada}.  

\begin{figure}[ht]
\begin{center}
\begin{overpic}[width=.08\textwidth]{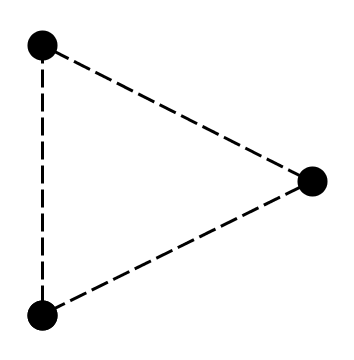}
\end{overpic}
\quad
\raisebox{0.5cm}
{
	\scalebox{1.5}{$\longleftrightarrow$}
}
\quad
\begin{overpic}[width=.08\textwidth]{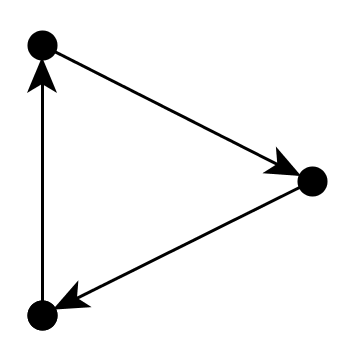}
\end{overpic}
\qquad \qquad
\begin{overpic}[width=.08\textwidth]{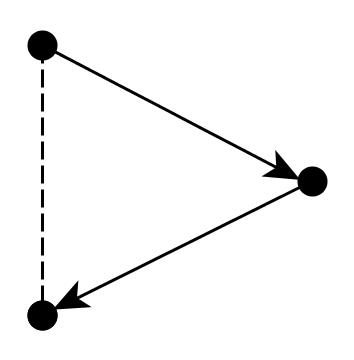}
\put(-17,85){u}
\put(-17,5){w}
\put(100,45){v}
\end{overpic}
\quad
\raisebox{0.5cm}
{
	\scalebox{1.5}{$\longleftrightarrow$}
}
\quad
\begin{overpic}[width=.08\textwidth]{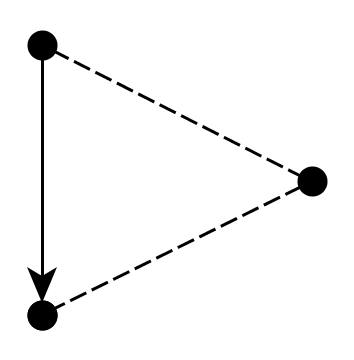}
\put(-17,85){u}
\put(-17,5){w}
\put(100,45){v}
\end{overpic}
\caption{The two net-degree preserving operations for digraphs.}
\label{fig:dir-ops}
\end{center}
\end{figure}

\begin{thm}\label{thm:dir-ops}
    Two digraphs $D_1 = (V,A_1)$ and $D_2 = (V,A_2)$ have the same
    net-degree sequence if and only if  $D_1$ can be obtained from $D_2$ by a
    sequence of $\Delta$- and $\Lambda$-operations.
\end{thm}

From this result, we obtain a simple characterization of those digraphs
corresponding to uniquely realizable net-degree sequences.

\begin{cor} \label{cor:uniq_rel_digraph}
    The degree sequence of a simple digraph $D$ is uniquely realizable if and
    only if $D$ does not contain one of the following four graphs as induced
    subgraph:
    \begin{center}
        \begin{overpic}[width=2cm]{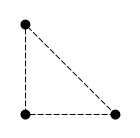}
            \put(27,-5){\rm (D1)}
        \end{overpic}
        \quad
        \begin{overpic}[width=2cm]{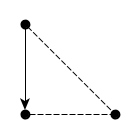}
            \put(27,-5){\rm (D2)}
        \end{overpic}
        \quad
        \begin{overpic}[width=2cm]{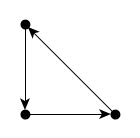}
            \put(27,-5){\rm (D3)}
        \end{overpic}
        \quad
        \begin{overpic}[width=2cm]{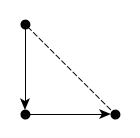}
            \put(27,-5){\rm (D4)}
        \end{overpic}
    \end{center}
\end{cor}

\begin{proof}
    A digraph does not permit a $\Delta$- or $\Lambda$-operation if and only
    if it avoids the given four induced subgraphs. The claim now follows from 
    Theorem~\ref{thm:dir-ops}.
\end{proof}

A digraph $D = (V,A)$ is \Defn{transitive} if whenever a node $v$ can be
reached from $u \in V$ by a directed path, then $(u,v)$ is a directed edge in
$D$. Digraphs $D = (V,A)$ that are both transitive and acyclic correspond exactly to 
partial order relations on $V$ and we simply call $D$ a \Defn{poset}. A
poset $D$ is of \Defn{width $\boldsymbol 2$} if every antichain, that is every
collection of nodes not reachable from one another, is of size at most $2$.

\begin{thm}\label{thm:uniq_digraph}
    A digraph $D = (V,A)$ has a uniquely realizable net-degree sequence if and
    only if $|V| \le 2$ or $D$ is a connected, width-$2$ poset.
\end{thm}
\begin{proof}
    If $D = (V,E)$ contains at most $2$ nodes it is clearly uniquely
    realizable and for a connected, width-$2$ poset it is easy to see that it
    satisfies the conditions of Corollary~\ref{cor:uniq_rel_digraph}.

    Conversely, let $D = (V,A)$ be a uniquely realizable digraph.  By using
    (D4) from Corollary~\ref{cor:uniq_rel_digraph} and induction on the length
    of the path, it follows that $D$ is transitive. Combining (D3) and (D4)
    shows that $D$ cannot contain cycles and hence is a poset. If $|V| > 2$,
    then (D2) excludes distinct components. Finally, (D1) assures us every
    antichain of $D$ is of size at most $2$. Hence $D$ is a connected,
    width-$2$ poset.
\end{proof}

Using this bijection, one can enumerate uniquely realizable degree sequences:

\begin{thm} \label{thm:width2-counting}
    Let $U(n)$ be the number of uniquely realizable net-degree sequences of
    simple digraphs  on $n \ge 0$ labelled vertices. Then
    \[
        U(n) \ = \ 
        \frac{n!}{\sqrt{3}}\Bigl( 
        (\sqrt{3}-1)^{-n-1}  - 
        (-\sqrt{3}-1)^{-n-1}\Bigr),
        \]
    with exponential generating function
    \[
        \sum_{n \geq 0} \frac{ U(n)x^n}{n!} \ = \ \frac{1}{1-x - \frac{x^2}{2}
        }. 
    \]
\end{thm}
\begin{proof}
    From Theorem~\ref{thm:uniq_digraph}, we infer that $U(0) = U(1) = 1$ and
    $U(2) = 3$. For $n \ge 3$, $D = (V,A)$ is a connected, width-$2$ poset.
    Let $M \subseteq V$ be the set of maximal elements. It follows from the
    definition of width-$2$ posets that $M$ contains either one or two
    elements. Removing $M$ from $D$ yields a uniquely realizable digraph on
    the node set $V\setminus M$. Hence, we infer that
    \[
        U(n) \ = \ n \cdot U(n-1) + {n \choose 2} \cdot  U(n-2) 
    \]
    for all $n \ge 2$. Now, define $\bar{U}(n) := \frac{1}{n!}U(n)$. Then
    $\bar{U}(n)$ satisfies the linear recurrence relation
    $\bar{U}(n) = U(n-1) + \frac{1}{2} U(n-2)$. The generating function as
    well as the formula can now be inferred using~\cite[Theorem~4.1.1]{EC1}.
\end{proof}

The first values of $U(n)$ are $1,1,3,12,66,450,3690$. Consistent with our
interpretation, $U(n)$ occurs in connection with ordered partitions of $[n]$
with parts of size at most $2$ and Boolean intervals in the weak Bruhat order
of $S_n$; see~\cite{Sloane} and Proposition~\ref{prop:width2-face}.

In terms of geometry, it is clear that every vertex of $\Pdir_n$ corresponds to a
uniquely realizable degree sequence. Indeed, every tournament is uniquely
determined by its net-degree sequence. In the language of posets, tournaments
correspond to total (or linear) orders and consequently have no incomparable
elements. Taking into account the number of pairs of incomparable elements, we
can locate the uniquely realizable net-degree sequences in $\Pdir_n$.

\begin{prop} \label{prop:width2-face}
    Let $\ds \in \Pdir_n \cap \Z^n$ be a net-degree sequence and let $F
    \subseteq \Pdir_n$ be the inclusion-minimal face that contains $\ds$ in its
    relative interior. Then the following statements are equivalent:
    \begin{enumerate}[\rm (i)]
        \item $\ds$ is uniquely realizable. Its corresponding poset has $k$
            pairs of incomparable elements.
        \item $F$ is a proper face lattice-isomorphic to the $k$-dimensional
            cube $[-1,1]^k$ for some $k \geq 0$. 
	\end{enumerate}
\end{prop}
\begin{proof}
    The polytope $\Pdir_n$ is lattice-isomorphic to the permutahedron $2
    \Pi_{n-1}$. The $k$-dimensional faces of $\Pi_{n-1}$ are in bijection to
    ordered partitions of $[n]$ into $n-k$ non-empty parts; \cite[Ch.~0]{Zie}.
    For an ordered partition $B = (B_1,\dots,B_{n-k})$ with sizes $a_i =
    |B_i|$, the corresponding face is lattice-isomorphic to $\Pi_{a_1-1}
    \times \cdots \times \Pi_{a_{n-k}-1}$. Hence, with
    Theorem~\ref{thm:uniq_digraph}, it suffices to prove that width-$2$ posets
    with $k$ pairs of incomparable elements are in bijection with ordered
    partitions of $[n]$ into $n-k$ parts. The corresponding poset $P$ is graded
    and the partition corresponds exactly to the rank partition of $P$.
\end{proof}

\section{Bidirected graphs}
\label{sec:bidirected}

Let $B = (G,\tau)$ be a bidirected graph as defined in the introduction. As
before, we will assume throughout that $G = (V,E)$ is an undirected graph
without loops or parallel edges on the node set $V = \{1,\dots,n\}$. The
bidirection function $\tau : V \times E \rightarrow \{-1,0,+1\}$ gives every
node-edge-incidence an orientation. Let us write $\ds(B)^+_i := |\{ e \in E
: \tau(i,e) = +1 \}|$ for the number of edges locally oriented towards $i \in [n]$ and
define $\ds(B)^-_i$ likewise. The \Defn{net-degree sequence} $\ds(B)$ of $B$
then satisfies $\ds(B)_i = \ds(B)_i^+ - \ds(B)^-_i$ for all $i=1,\dots,n$. 

Defining orientations of edges locally renders net-degree sequences of
bigraphs far less restricted as in the undirected or directed setting. A
characterization is nevertheless not vacuous.  Indeed, a key observation is
that the underlying undirected graph $G = |B|$ has ordinary degree sequence
$\ds(G) = \ds(B)^+ + \ds(B)^-$ and hence
\begin{equation}\label{eqn:Bkey}
    \ds(G)_i \ \equiv \ \ds(B)_i \ \ \mathrm{mod}\ 2  \quad \text{ and } \quad
    \ds(G)_i \ \geq \  |\ds(B)_i|, 
\end{equation}
for all $i=1,\dots,n$.  Borrowing from Section~\ref{sec:undirected}, we write
$\ext{\ds} = (\ds,\frac{1}{2}(d_1+\cdots+d_n)) \in \R^{n+1}$. Towards a
geometric perspective on degree sequences of bigraphs, we define the polytope
\[
    \Pbid_n \ := \ \conv \Bigg\{ \sum_{i < j} \tau_{ij} \ext{e}_i + \tau_{ji}
    \ext{e}_j  : \tau \in \{-1,+1\}^{n \times n} \Bigg\}.
\]

\begin{thm}\label{thm:bErdGal}
    For a vector $\ds = (d_1,\dots,d_n) \in \Z^n$ the following are equivalent:
    \begin{enumerate}[\rm (i)]
        \item There is a bigraph $B$ with $\ds = \ds(B)$;
        \item $\ext{\ds} \in \Pbid_n$; \item $d_1 + \cdots + d_n$ is even and
            $-(n-1) \leq d_i \leq n-1$ for $i=1, \ldots, n$.
    \end{enumerate}
\end{thm}
\begin{proof}
    Let $\ds = \ds(B)$ be the net-degree sequence of a bigraph $B$.  The
    net-degree sequence is given by
    \[
        \ds(B) \ = \ \sum_{i \in V, e \in E} \tau(i,e) e_i
    \]
    and~\eqref{eqn:Bkey} together with~\eqref{eqn:d_even} shows that
    $\ext{\ds} \in \Pbid_n$.  This proves (i) $\Longrightarrow$ (ii). The
    implication (ii) $\Longrightarrow$ (iii) is obvious.

    For (iii) $\Longrightarrow$ (i), we first observe that $T = \{ i \in [n] :
    d_i \not\equiv n-1 \ \mathrm{mod}\ 2\}$ has an even number of elements
    which follows from the fact that $d_1 + \cdots + d_n \equiv n(n-1) - |T|$
    modulo $2$.  Let $G = ([n],E)$ be the graph obtained from the complete
    graph $K_n$ by removing a matching of the elements of $T$. Hence, $G$ is a
    simple undirected graph with ordinary degree sequence $\ds' := \ds(G)$.
    The degrees satisfy $\ds(G)_i = n-2$ if $i \in T$ and $\ds(G)_i= n-1$, otherwise.
    We can now turn $G$ into a bigraph $B$ by bidirecting its edges. For every
    node $i \in [n]$, we let $|\ds_i|$ edges point to $i$ or away from $i$
    depending on the sign of $\ds_i$. By construction, the number of
    unoriented edges incident to $i$ is even and we can orient them so that
    their contribution to $\ds(B)_i$  cancels. See Figure~\ref{FigExBOfD} for
    an illustration.
\end{proof}

\begin{figure}[ht]
\centering
\begin{overpic}[width=.1\textwidth]{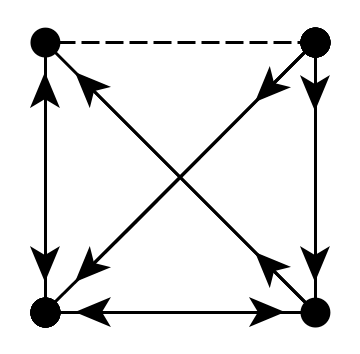} 
\put(-18,84){$v_1$}
\put(-18,5){$v_3$}
\put(100,84){$v_2$}
\put(100,5){$v_4$}
\end{overpic}
\caption{Realizing $\ds=(2,-2,3,1)$ by a bigraph.}
\label{FigExBOfD}
\end{figure}
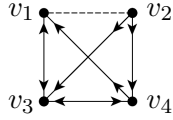

Theorem~\ref{thm:bErdGal} gives us a means to count degree
sequences.

\begin{cor}\label{cor:bdir_count}
    The number of net-degree sequences of bigraphs on $n$ labelled
    vertices is  
    $
        \frac{(2n-1)^n+1}{2}.
    $
\end{cor}
\begin{proof}
    The number of lattice points in the cube $[-(n-1),(n-1)]^n$ is
    $(2(n-1)+1)^n$ and according to Theorem~\ref{thm:dErdGal}(iii) those with
    even coordinate sum are in bijection with degree sequences of bigraphs.
    Now if $u \in [-(n-1),(n-1)]^n$ has even coordinate sum, then so has $-u$
    and $u \neq -u$ unless $u = 0$. 
    Hence, there are $\frac{(2(n-1)+1)^n-1}{2} + 1
    $ lattice points with even coordinate sum.
\end{proof}

Bigraphs simultaneously generalize undirected graphs and digraphs. Hence, it
is natural to assume that there are more degree-preserving operations for
bigraphs.  The most basic operation that is not present in the directed or
undirected setting is the \Defn{$\boldsymbol \Gamma$-operation}: A
\mbox{$\Gamma$-operation} swaps the local orientations of two edges $uv$ and $vw$
incident to a node $v$ with $\tau(v,uv)\neq\tau(v,vw)$;
see~Figure~\ref{FigureBidirectionChange}. In particular, a $\Gamma$-operation
leaves the underlying graph $|B|$ untouched and only changes the bidirection
function $\tau$.

\begin{figure}[ht]
\centering
\begin{overpic}[width=.1\textwidth]{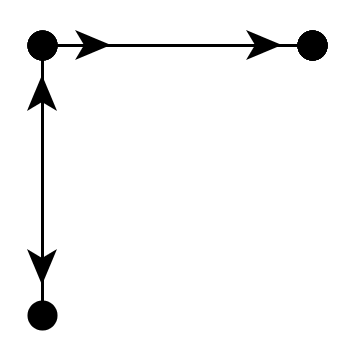}
\put(-10,85){v}
\put(-12,5){w}
\put(100,85){u}
\end{overpic}
\quad
\raisebox{0.75cm}
{
	\scalebox{1.5}{$\longleftrightarrow$}
}
\quad
\begin{overpic}[width=.1\textwidth]{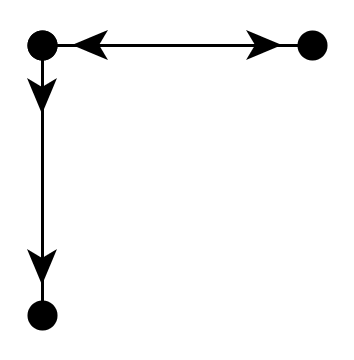}
\put(-10,85){v}
\put(-12,5){w}
\put(100,85){u}
\end{overpic}
\caption{The $\Gamma$-operation.}
\label{FigureBidirectionChange}
\end{figure}
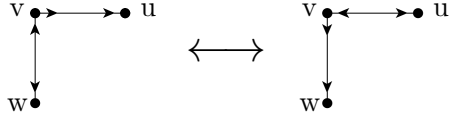

For two bigraphs with the same degree sequence and the same underlying graph,
the \mbox{$\Gamma$-operation} is sufficient to transform one into the other.

\begin{lem} \label{lem:bid_Gamma}
    Let $B$ and $B'$ be two bigraphs with $|B| = |B'|$ and $\ds(B) = \ds(B')$.
    Then there is a finite sequence of $\Gamma$-operations that transforms $B$
    into $B'$. 
\end{lem}
\begin{proof}
    Let $B = (G,\tau)$ and $B' = (G,\tau')$ with underlying undirected graph
    $G = (V,E)$.  For a fixed node $v \in V$, the number of edges $e \in E$
    with $\tau(v,e) \neq \tau'(v,e)$ is even. Hence, we can pair up these
    edges and apply a $\Gamma$-operation to $B$ for every pair. Since the
    changes only affect the local orientations at $v$, we may repeat this
    process for all the remaining nodes.
\end{proof} 

The $\Sigma$- or $2$-switch operation on undirected graphs has a natural
extension to bidirected graphs: A $\Sigma$-operation locally replaces an edge
$e$ incident to a node $v$ with a new incident edge $e'$. If $G$ is the graph
underlying a bidirected graph $B$, then a $\Sigma$-operation applied to $B$
assigns a local orientation to $e'$ by $\tau(v,e') := \tau(v,e)$; see
Figure~\ref{FigureOriented2Switch} for a depiction. Note that a
$\Sigma$-operation on $B$ induces a $\Sigma$-operation on the underlying
undirected graph $|B|$. Combining this observation with the Havel--Hakimi
theorem~\cite{Havel,Hakimi} yields the following corollary.

\begin{cor}\label{cor:bdirHH}
    Let $B$ and $B'$ be bigraphs such that their underlying graphs $|B|$ and
    $|B'|$ have identical degree sequences. Then $B$ can be transformed using
    $\Sigma$-operations into a bigraph with underlying graph $|B'|$.
\end{cor}

\begin{figure}[ht]
\centering
\begin{overpic}[width=.1\textwidth]{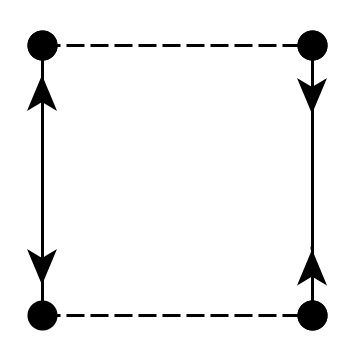}
\end{overpic}
\quad
\raisebox{0.75cm}
{
	\scalebox{1.5}{$\longleftrightarrow$}
}
\quad
\begin{overpic}[width=.1\textwidth]{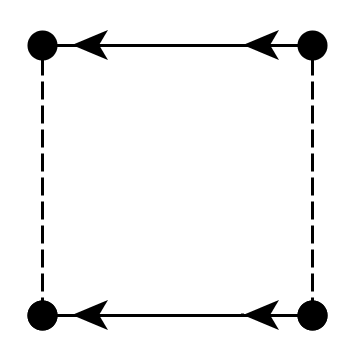}
\end{overpic}
\caption{The $\Sigma$-operation.}
\label{FigureOriented2Switch}
\end{figure}

As for directed graphs, an edge $uw$ in a bigraph $B$ may be replaced by edges
$uv, vw$ where $v\in V$ is a node not adjacent to either $u$ or $w$. This
procedure is net-degree preserving if we set $\tau(v,uv) = -\tau(v,vw) = 1$.
This operation and its inverse is called a \Defn{$\boldsymbol \Lambda$-operation};
Figure~\ref{FigureDefPathAbbr} illustrates the operation.

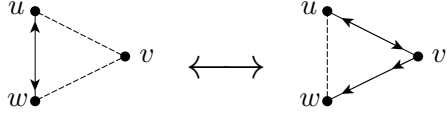
\begin{figure}[ht] 
\centering
\begin{overpic}[width=.1\textwidth]{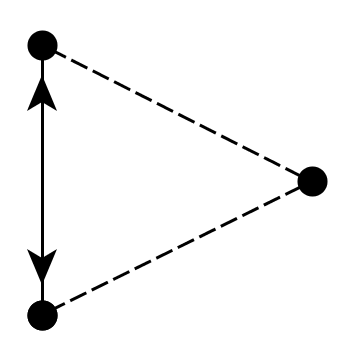}
\put(-11,85){$u$}
\put(-12,8){$w$}
\put(99,45){$v$}
\end{overpic}
\quad
\raisebox{0.5cm}
{
	\scalebox{1.5}{$\longleftrightarrow$}
}
\quad
\begin{overpic}[width=.1\textwidth]{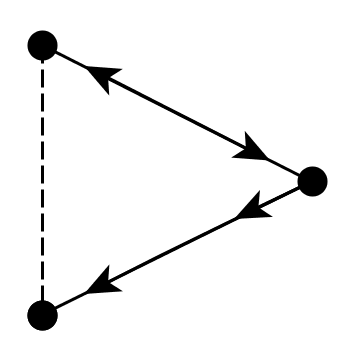}
\put(-11,85){$u$}
\put(-12,8){$w$}
\put(99,45){$v$}
\end{overpic}
\caption{The $\Lambda$-operation.}
\label{FigureDefPathAbbr}
\end{figure}

Finally, we extend our repertoire of operations by a
\Defn{$\boldsymbol\Delta$-operation} that corresponds to adding or deleting a
bidirected triangle: If $u_1,u_2,u_3$ are vertices of a bigraph $B$ with edges
$u_1u_2, u_2u_3,$ $u_3u_1$ such that $\tau(u_i,u_iu_j) \neq \tau(u_j,u_iu_j)$ for
$i\neq j$, then we may delete the three edges without changing the net-degree
of $B$. 

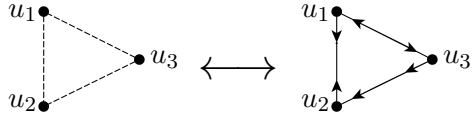
\begin{figure}[ht] 
\centering
\begin{overpic}[width=.1\textwidth]{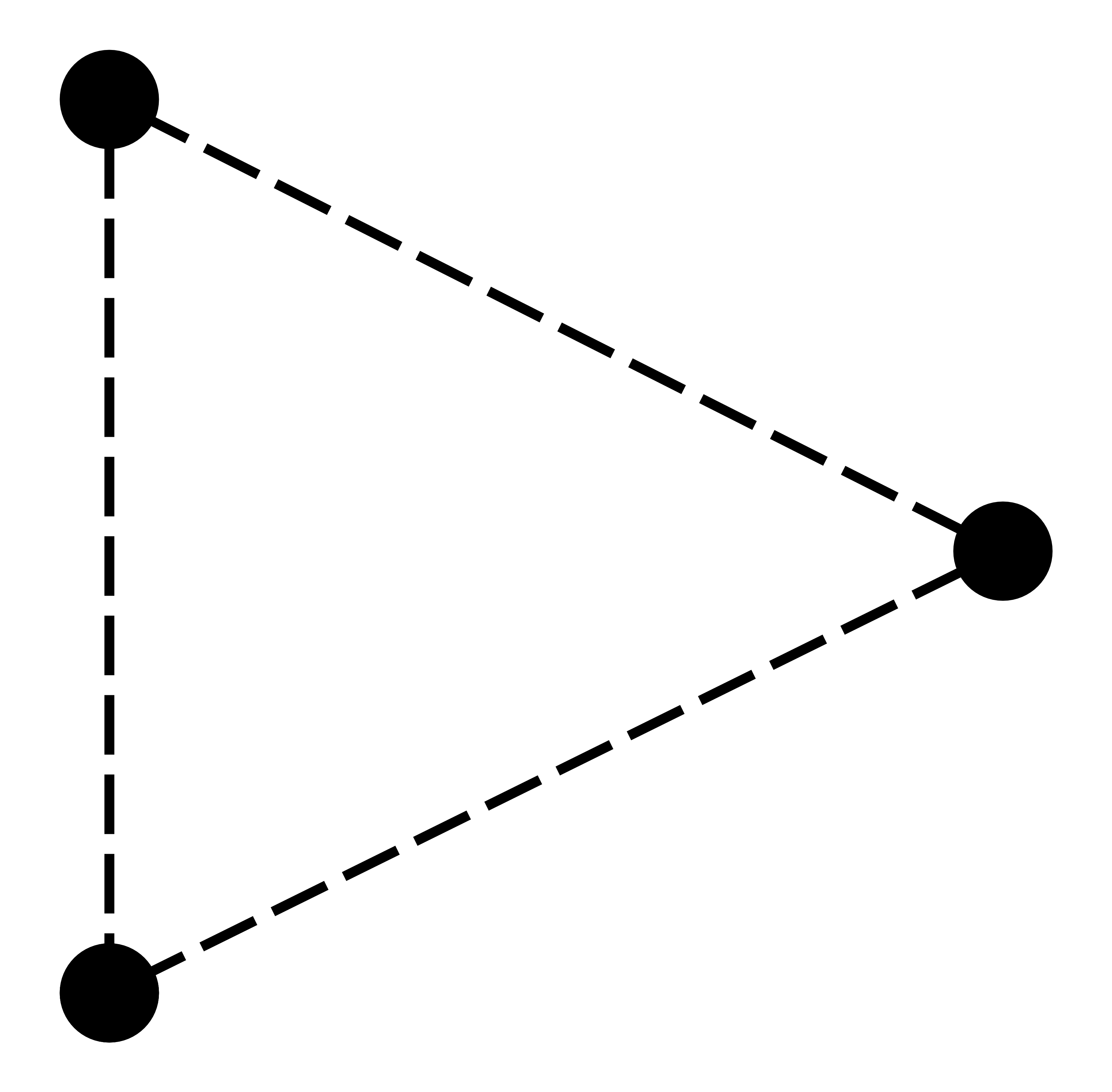}
\put(-21,85){$u_1$}
\put(-21,8){$u_2$}
\put(100,45){$u_3$}
\end{overpic}
\quad
\raisebox{0.5cm}
{
	\scalebox{1.5}{$\longleftrightarrow$}
}
\quad
\begin{overpic}[width=.1\textwidth]{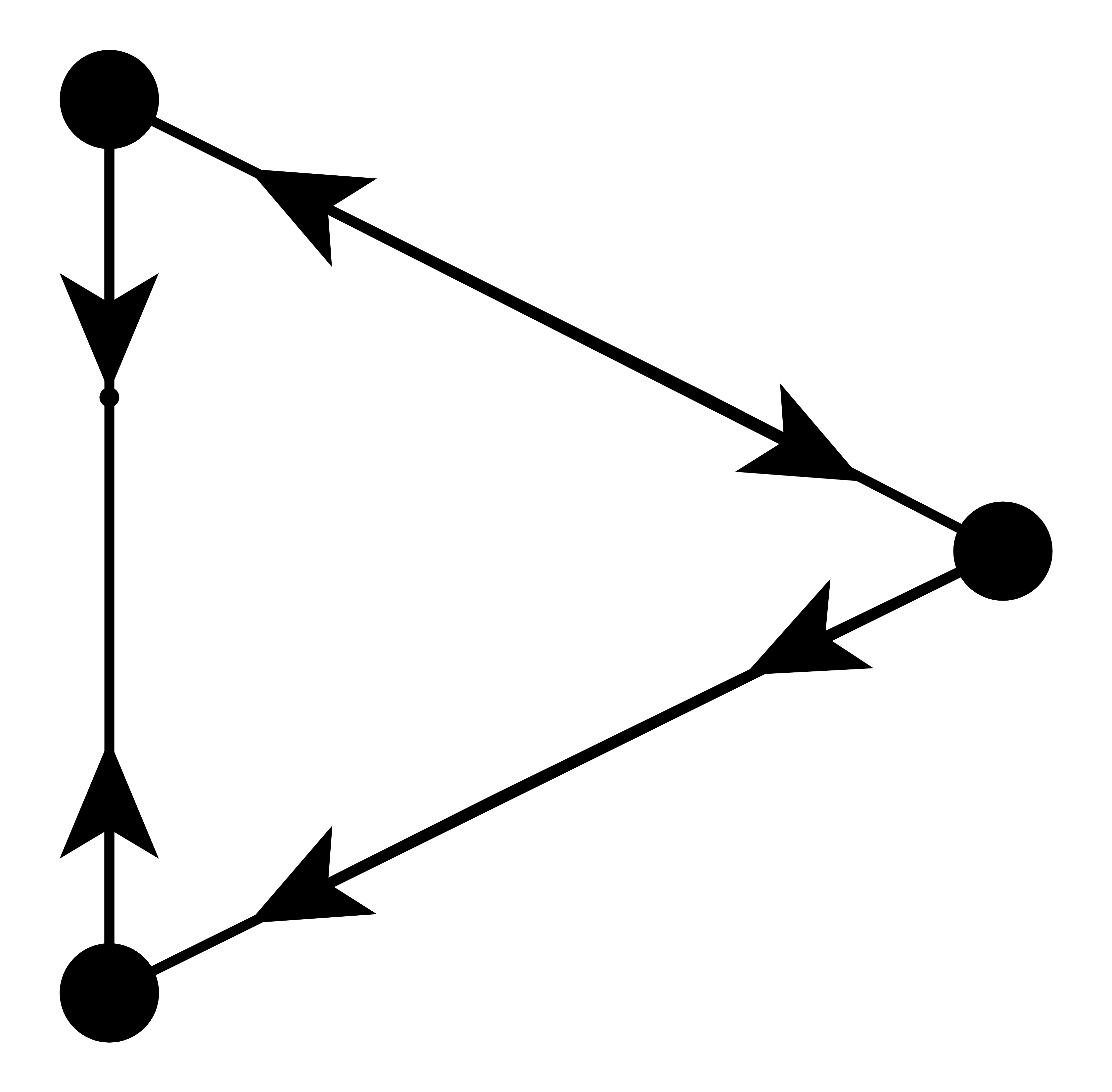}
\put(-21,85){$u_1$}
\put(-21,8){$u_2$}
\put(100,45){$u_3$}
\end{overpic}
\caption{The $\Delta$-operation.}
\label{FigureDefDelta-Op}
\end{figure}

More generally, a \Defn{bidirected $\boldsymbol k$-cycle} in a bidirected
graph $B = (G,\tau)$ is an undirected cycle $u_1,\dots,u_k$ in $G$ such that
$\tau(u_i,u_{i-1}u_i) \neq \tau(u_i,u_iu_{i+1})$ for $i=2,\dots,k$ (and the
convention that $u_{k+1} := u_1$). Clearly, the net-degree of bidirected
cycles is identically zero and we could extend $\Delta$-operations to all
bidirected cycles. The next lemma, however, shows that this can always be
achieved essentially by $\Delta$- and $\Lambda$-operations.

\begin{lem} \label{lem:bi-cycle-addition}
    Let $B$ be a bidirected graph. Then adding a bidirected cycle can be obtained
    by a sequence of $\Delta$-, $\Lambda$-, and $\Gamma$-operations.
\end{lem}
\begin{proof}
    Let us assume that we want to add a cycle of length $m$ on the nodes
    $1,2,\dots,m$. To ease notation, we identify $m+1$ with $1$.  We prove the
    claim by induction on $m$ with $m=3$ as the base case. If $i,i+2$ is not
    an edge for some $i=1,\dots,m-1$, then we can add the cycle $1, \dots, i,
    i+2, \dots, m$ by induction and apply a $\Lambda$-operation to replace the
    edge $i,i+2$ by $i,i+1$ and $i+1,i+2$. Hence, we have to assume that all
    edges $i,i+2$ for $i=1,\dots,m-1$ are present.  If $m=4$, we replace the
    edge between $1$ and $3$ by the edges $1,2$ and $2,3$. Then, we add the
    triangle between $1, 3$ and $4$.  If $m \geq 5$,  then for all vertices
    $2j$ with $j = 1,\dots,\lfloor\frac{m}{2}\rfloor$, we use a
    $\Lambda$-operation to replace the edge $2j-1,2j+1$ with the path
    $2j-1,2j,2j+1$ without changing the local orientation at $2j\pm 1$.  We
    may now add the cycle on the vertices $1,3,\dots,2\lceil\frac{m}{2}\rceil
    -1$. 

    The resulting bigraph $B'$ has already the correct underlying undirected
    graph and degree sequence $\ds(B') = \ds(B)$ and appealing to
    Lemma~\ref{lem:bid_Gamma} completes the proof.
\end{proof}

An argument analogous to the above shows that we can add arbitrary bidirected
paths using essentially $\Lambda$-operations. A \Defn{bidirected path} in a
bidirected graph $B = (G,\tau)$ is an undirected path $u_1,\dots,u_k$ in $G$
such that $\tau(u_i,u_{i-1}u_i) \neq \tau(u_i,u_iu_{i+1})$ for
$i=2,\dots,k-1$.

\begin{lem}\label{lem:bi-path-addition}
    Let $B$ be a bidirected graph. Then adding a bidirected path can be obtained
    by a sequence of $\Lambda$- and $\Gamma$-operations.
\end{lem}

The following theorem asserts that  $\Gamma$-, $\Sigma$-, $\Lambda$-, 
and $\Delta$-operations suffice to navigate
among bigraphs with the same net-degree sequence. The bigraphs in
Figures~\ref{FigureBidirectionChange}, \ref{FigureDefPathAbbr},
and~\ref{FigureDefDelta-Op} show that the $\Gamma$-, $\Lambda$-, and
$\Delta$-operations are necessary. To see that the $\Sigma$-operation is
necessary it suffices to consider $4$-cycles where every edge points into its two end nodes.

\begin{thm} \label{thm:bid_Ops}
    Let $B$ and $B'$ be bigraphs with $\ds(B) = \ds(B')$. Then there is
    a finite sequence of $\Gamma$-, $\Sigma$-, $\Lambda$-, and
    $\Delta$-operations that transforms $B$ into $B'$.
\end{thm}

\begin{proof}
    Let $B = (G,\tau)$ and $B' = (G',\tau')$ on $n$ nodes.  The
    strategy of the proof is as follows. We show that we can transform both
    $B$ and $B'$ into bigraphs $B_0$ and $B'_0$ such that their underlying
    undirected graphs $G_0$ and $G'_0$ have identical (ordinary) degree
    sequences. Corollary~\ref{cor:bdirHH} then shows that we can assume that
    $G_0 = G'_0$ and Lemma~\ref{lem:bid_Gamma} completes the proof.

    Let $T := \{i \in [n] : \ds(B)_i \not\equiv n-1 \ \mathrm{mod} \ 2\}$ be
    the set of odd-degree vertices in $H := K_n \setminus G$. Since $T$ is of even
    size, we can label its elements $v_1, v_2,\dots, v_{2m}$.  The edge set
    $E(H)$ can be decomposed into edge-disjoint cycles $C_1,\dots,C_s$ and
    paths $P_1,\dots,P_m$ such that, without loss of generality,  $P_i$ has
    endpoints $v_{2i-1}$ and $v_{2i}$; see~\cite[Section 29.1]{Schri}.  For
    every $j=1,\dots,s$, we apply Lemma~\ref{lem:bi-cycle-addition} to add a
    bidirected cycle to $B$ with underlying graph $C_j$.  By the same token,
    if $P_i$ is of length $> 1$, then we can replace the bidirected edge
    between $v_{2i-1}$ and $v_{2i}$ in $B$ by a bidirected path with
    underlying graph $P_i$ using Lemma~\ref{lem:bi-path-addition}. The
    resulting bigraph $B_0$ has net-degree sequence $\ds(B)$ and, more
    importantly, every node $v$ in $|B_0| = G_0$ has ordinary degree $n-2$ if
    $v \in T$ and $n-1$ otherwise. Repeating this process for $B'$, we obtain
    our desired bigraphs $B_0$ and $B'_0$.
\end{proof}

A complete set of degree-preserving operations as given in
Theorem~\ref{thm:bid_Ops} also gives insight into the class of uniquely
realizable net-degree sequences. A bigraph has a uniquely realizable degree
sequence if and only if no degree-preserving operation is applicable. Recall
that a node $v$ is a sink or a source if all edges are pointing into $v$ or
away from $v$, respectively.

\begin{cor} \label{cor:bid_unique}
    The net-degree sequence of a bigraph $B$ is uniquely realizable if and
    only if every node is a sink or a source and there is at most one pair
    of nodes not joined by an edge.

    In particular, a net-degree sequence $\ds = (d_1,\dots,d_n) \in \Pbid_n
    \cap \Z^n$ is uniquely realizable if and only if $n-2 \le |d_i| \le n-1$
    for all $i=1,\dots,n$ and there at most two entries with absolute value
    $n-2$. 
\end{cor}

In particular, the set of uniquely realizable net-degree sequences is
contained in the boundary of $\Pbid_n$. The vertices of $\Pbid_n$ are uniquely
realizable as in the case of ordinary and directed graphs.

\begin{cor}\label{cor:bid_uniq_num}
    The number of uniquely realizable net-degree sequences of bidirected
    graphs on $n$ nodes is
    \[
        2^n\binom{n}{2} + 2^n.
    \]
\end{cor}

\end{document}